\documentclass[journal]{IEEEtran}
\def\baselinestretch{1.2}

\usepackage{tikz}
\graphicspath{{images/}}
\usepackage{amsmath}
\usepackage{latexsym, amssymb}
\usepackage{graphicx}
\usepackage{amsmath, amsbsy}
\usepackage{amsopn, amstext}
\usepackage{ifpdf,hyperref}
\usepackage{cancel, color}
\usepackage{epstopdf}

\def\ttimes{\,\rotatebox[]{-90}{$\ltimes$}\,}

\def\J{{\bf 1}}

\DeclareMathOperator{\Span}{Span}
\DeclareMathOperator{\Col}{Col}

\DeclareMathOperator*{\argmin}{argmin}
\DeclareMathOperator{\lcm}{lcm}

\def\cal{\mathcal}
\def\T{{\bf T}}

\def\diag{\mathrm{diag}}
\def\ra{\rightarrow}

\def\lra{\leftrightarrow}

\def\V{{\cal V}}
\def\a{\alpha}
\def\b{\beta}
\def\d{\delta}
\def\e{\epsilon}

\def\l{\lambda}

\def\0{{\bf 0}}

\def\haa{\longleftrightarrow}
\newcommand{\vaa}{\mathbin{\text{\rotatebox[origin=c]{90}{$\longleftrightarrow$}}}}
\newcommand{\R}{{\mathbb R}}

\newcommand{\C}{{\mathbb C}}
\newcommand{\Z}{{\mathbb Z}}
\newcommand{\F}{{\mathbb F}}

\def\dsum{\mathop{\sum}\limits}

\newtheorem{thm}{Theorem}[section]
\newtheorem{dfn}[thm]{Definition}
\newtheorem{prp}[thm]{Proposition}
\newtheorem{exa}[thm]{Example}
\newtheorem{lem}[thm]{Lemma}
\newtheorem{cor}[thm]{Corollary}
\newtheorem{rem}[thm]{Remark}
\newtheorem{alg}[thm]{Algorithm}

\begin{document}

\title{From Eigenvalues/Eigenvectors of Hypermatrices to Canonical Form of Tensors}

\author{Daizhan Cheng,~June Feng,~Hongsheng Qi$^* $,~Jiandong Zhu

\thanks{This work is supported partly by the National Natural Science Foundation of China (NSFC) under Grants 62350037.}
\thanks{Daizhan Cheng is with Key Laboratory of Systems and Control, Academy of Mathematics and Systems Sciences, Chinese Academy of Sciences,
		Beijing 100190; and Research Center of Semi-tensor Product of Matrices-Theory and Applications, Liaocheng, P. R. China (e-mail: dcheng@iss.ac.cn).}
\thanks{June Feng is with Dept. of Mathematics, Shandong University, Jinan, P.R. China, (e-mail: fengjune@adu.edu.cn).}
\thanks{Hongsheng Qi is with the State Key Laboratory of Mathematical Sciences, Academy
of Mathematics and Systems Science, Chinese Academy of Sciences, Beijing
100190; and the School of Mathematical Sciences, University of Chinese Academy of Sciences, Beijing 100049, P. R. China (e-mail: qihongsh@amss.ac.cn). $~^* $: Corresponding Author.
}
\thanks{Jiandong Zhu is with School of Mathematical Sciences,  Nanjing Normal University, Nanjing, P.R. China (e-mail:zhujiandong@njnu.edu.cn).}
}
\maketitle

\begin{abstract}
The rings of non-square matrices based on the dimension-keeping (DK-)
semi-tensor product (STP) are considered, where a virtual identity is introduced
to make each ring possess an identity element. Using this ring structure, four
kinds of eigenvalues/eigenvectors (EEs) of hypermatrices, namely, the ordinary EE
(OEE), the universal EE (UEE), the diagonal EE (DEE), and the horizontal diagonal
EE (HDEE), are proposed with respect to preassigned matricizings. The Kronecker
canonical form (KCF) of non-square pencils is used to calculate the OEEs; the
monic decomposition algorithm (MDA) is then applied to extract the UEEs, DEEs,
and HDEEs from them. Finally, the KCF of non-square pencils is further used to
construct the KCF of a tensor, which reveals all the EEs of the tensor. The KCF
of a tensor not only shares the main properties of the Jordan canonical form of
matrices but also includes the latter as a special case. Consequently, all the
EEs of a hypermatrix are straightforwardly computable via its KCF.
\end{abstract}


\IEEEpeerreviewmaketitle

\section{Introduction}

Let $A\in {\cal M}_{n\times n}$ be a square matrix. A nonzero vector $x\in \C^n$ is called an
eigenvector of $A$ associated with the eigenvalue $\lambda\in \C$, if
\begin{align}\label{1.1}
Ax=\lambda x.
\end{align}
If there exists a set of nonzero vectors $x_1,~x_2,~\cdots,~x_r$ such that
\begin{align}\label{1.2}
\begin{array}{l}
Ax_1=\lambda x_1,\\
Ax_2=x_1+\lambda x_2,\\
Ax_3=x_2+\lambda x_3,\\
~~~~~~~\vdots\\
Ax_r=x_{r-1}+\lambda x_r,
\end{array}
\end{align}
then $\{x_1,\cdots,x_r\}$ is called a set of chained (generalized) eigenvectors associated with
the eigenvalue $\lambda$. To be precise, $x_1$ is an eigenvector, while $x_2,\cdots,x_r$ are
called root vectors (or generalized eigenvectors).

Eigenvalues and eigenvectors (EEs) play a fundamental role in matrix theory: they
essentially characterize the invariant subspaces of a square matrix. A square matrix
$$
J_r(\lambda)=\begin{bmatrix}
\lambda&1&~&~&~\\
~&\lambda&1&~&~\\
~&~&\lambda&\ddots&~\\
~&~&~&\ddots&1\\
~&~&~&~&\lambda\\
\end{bmatrix}\in {\cal M}_{r\times r}
$$
is called a Jordan block of size $r$. It is well known \cite{hor85} that, for a given matrix
$A\in {\cal M}_{n\times n}$, there exists a nonsingular matrix $P$ such that
\begin{align}\label{1.3}
P^{-1}AP=J,
\end{align}
where
$$
J=\diag\left( J_{n_1}(\lambda_1),\cdots,J_{n_s}(\lambda_s)\right),
$$
with $n_1+\cdots+n_s=n$. $J$ is called the Jordan canonical form of $A$. Split $P$ into
$s$ blocks as
$$
P=[P^1,P^2,\cdots,P^s],
$$
where $P^k\in {\cal M}_{n\times n_k}$, $k\in [1,s]$. It is then easy to see that
\begin{align}\label{1.4}
\{P^k_1,\cdots,P^k_{n_k}\}
\end{align}
is a set of chained eigenvectors associated with $\lambda_k$, $k\in [1,s]$. Consequently,
$V_k=\Span\{P^k_1,\cdots,P^k_{n_k}\}$, $k\in [1,s]$, are $A$-invariant subspaces of $\C^n$.
The preceding discussion makes it clear that the Jordan canonical form $J$, together with its
transformation matrix $P$, reveals complete information about the EEs and the invariant
subspaces of a matrix.

The general eigenvalue and eigenvector of a matrix $A\in {\cal M}_{m\times n}$ with respect to
a type $B\in {\cal M}_{m\times n}$ are defined by $\lambda\in \C$ and $0\neq x\in \C^n$, such
that
\begin{align}\label{1.5}
(A-\lambda B)x=0,
\end{align}
where both $m=n$ and $m\neq n$ are allowed. We refer to~\cite{van75} and the references
therein for detailed discussions. As shown in~\cite{che25}, certain subsets of such
eigenvectors may represent subspaces with various different behaviors: in addition to
invariant subspaces, they may characterize, for instance, subspaces perpendicular to the
original eigenspace.

The fundamental idea of this paper is to find all EEs of a hypermatrix by constructing its
canonical form.

In most AI-related literature, both hypermatrix and tensor are used to describe higher-order
data. According to \cite{lim13}, a hypermatrix is a mapping
$\pi:[1,n_1]\times\cdots\times [1,n_d]\ra \F$, where $\F$ is a field (to simplify the
argument, this paper assumes $\F=\R$ or $\F=\C$ only). Roughly speaking, a tensor of covariant
order $r$ and contravariant order $s$ over an $n$-dimensional vector space $V$, denoted by
$t^r_s(V)$, is a multilinear mapping \cite{boo86}. Under a fixed basis of $V$, $t^r_s(V)$ is
determined by a set of structure constants
$$
\left\{\left.\mu^{i_1,\cdots,i_r}_{j_1,\cdots,j_s}\;\right|\;
i_p,j_q\in [1,n],~p\in [1,r],~q\in [1,s]\right\},
$$
which can be regarded as a hypermatrix $\pi:[1,n]^{r+s}\ra \F$. Based on this observation,
this paper calls the hypermatrices with $n_1=\cdots=n_d$ tensors. Since a hypermatrix is
commonly expressed as a (mostly non-square) matrix, the EEs of hypermatrices are treated as
general EEs of matrices, while the canonical form is considered for tensors.

In 2004 the American Institute of Mathematics (AIM) organized a workshop on tensor
decompositions. The workshop discussion notes \cite{mar04} raised a series of open problems,
which can be summarized as one general open topic:
``How to extend the algebraic structure of linear algebra to tensors?''
The workshop was followed by an upsurge of research on tensors. This paper is particularly
interested in the EEs of hypermatrices, which are also closely related to the Kronecker
product decomposition (KPD). In 2005, Qi and Lim independently proposed EEs of tensors
\cite{qi05,lim05}. Since then, numerous definitions, properties, and applications have
emerged. The KPD of tensors has also been studied extensively \cite{bat17}. Another related
topic is the singular value decomposition (SVD) of tensors \cite{lat00}. Tensor-based
techniques have found wide applications in signal and image processing \cite{kil13}.

The monic decomposition algorithm (MDA) was first proposed in \cite{che25} for solving the EE
problem of hypermatrices. It was then used to solve the Kronecker product decomposition
problem \cite{che26c}, and has proved to be a powerful tool for investigating the internal
structure of hypermatrices.

Nowadays, along with the rise of large-scale AI foundation models, tensors have become a
prominent research subject.
It is claimed online \cite{xiu26} that ``All mainstream artificial intelligence essentially
performs one task: tensor computation. It is precisely this computation that gives machines
the appearance of intelligence.''
Recently, it was claimed in \cite{kle26} that ``In a transformer, tensor decomposition is used
to unify all attention mechanisms, and the whole computation of multi-head attention can be
expressed as the module product of two tensors.''
In one word, tensor, as well as hypermatrix, is an important tool in large-scale AI systems.

The purpose of this paper is threefold:
\begin{enumerate}
\item[(1)] to provide four kinds of EEs for hypermatrices, namely, the ordinary EE (OEE), the
    universal EE (UEE), the diagonal EE (DEE), and the horizontal diagonal EE (HDEE);
\item[(2)] to calculate the OEEs using the Kronecker canonical form (KCF) of matrix pencils,
    and then to apply the MDA to verify whether an OEE belongs to one of the other
    three special cases; and
\item[(3)] to propose a Jordan-like canonical form, called the KCF of tensors, and to use it
    to reveal the relationship between the EEs of a tensor and its KCF, in analogy to that
    between the EEs and the Jordan canonical form of a matrix.
\end{enumerate}

The OEE is precisely the general EE~\eqref{1.5} of a matricized hypermatrix. The UEE was
proposed in \cite{che25}, in which the eigenvectors are decomposable. It was
pointed out there that (a) to the best of the authors' knowledge, all existing definitions of
EEs of hypermatrices are special cases of the UEE; and (b) the eigenvectors of the UEE carry
significant geometric meaning, representing, for instance, invariant subspaces and
perpendicular subspaces. The DEE is a special case of the UEE, in which the component vectors
of the decomposable eigenvector are identical. Finally, the HDEE corresponds to a matricizing
with a single row index, and it covers all the definitions existing before \cite{che25}.
Hence, the definitions of these EEs are not new per se; they have merely been reorganized
into a nested hierarchy:
\begin{align}\label{1.6}
\text{OEE}\supset \text{UEE}\supset \text{DEE} \supset \text{HDEE}.
\end{align}

The calculation of the OEEs is based on the KCF of matrix pencils \cite{gan59,van79}. Using
the MDA, the UEEs can then be selected from the OEEs; similarly, the DEEs and HDEEs can be
obtained. Since the KCF of a matrix pencil is straightforwardly computable, it follows that
all the EEs of a hypermatrix are also easily computable. Indeed, solving the EE problem of
hypermatrices directly requires solving a system of polynomial equations in multiple unknowns,
which is in general extremely difficult, even numerically. In contrast, the method developed
in this paper involves only the computation of the KCF, which requires nothing more than
elementary row and column operations, together with a Jordan form transformation of a square
submatrix. This dramatically simplifies the calculation of the EEs of hypermatrices.

Finally, we would like to emphasize that our definition relies on the matricization of
hypermatrices, which we believe is reasonable. Even in the second-order case, eigenvalues and
eigenvectors are associated with matrix representations ($A$ and $A^{\T}$, or even vector
forms for general EEs). Furthermore, every matricization retains the full information of the
tensor: as an invertible transformation, it does not degrade the intrinsic structure or
information of the tensor.

The rest of this paper is organized as follows:
Section 2 reviews the DK-STP and considers the ring of the set of $m\times n$ matrices. A virtual identity is introduced to make the ring unitary (i.e., with product identity). Section 3 investigates some fundamental problems for hypermatrices: (a) the matricizing and vectorizing of hypermatrices; (b) EEs of hypermatrices; (c) monic decomposition algorithm for hypervectors. The KCF of matrix pencil and its relationship with general EE of matrices are discussed in Section 4. Section 5 extends the KCF of matrix pencil to hypermatrices and proposes a method to calculate EEs of a hypermatrix: first using KCF to find OEEs and the using MDA to pick out UEEs, DEEs, and HDEEs. Finally, the KCF of tensors is obtained. Section 6 investigates firstly the equivalence of tensors with its relationship with their OEEs. Then the equivalence propagation from one matricizing to another matricizing.
Section 7 is a brief conclusion.

Before ending this section a list of notations is presented.

\begin{itemize}

\item $\R$: set of real numbers.

\item $\C$: set of complex numbers.

\item $\Z_+$: set of positive integers.

\item $\lcm(a,b)$ (or $a\vee b$): least common multiple of $a$ and $b$.

\item $\gcd(a,b)$ (or $a\wedge b$): greatest common divisor of $a$ and $b$.

\item $\J_{n}$ ($\J_{m\times n})$: A vector of dimension $n$ (a matrix of dimension $m\times n$) with entries $1$ only.

\item $0_n$: $0_n\in \R^n$ is the $n$-dimensional zero vector.



\item ${\cal M}_{m\times n}$: set of $m\times n$-dimensional real or complex matrices.

\item$\otimes$: Kronecker product of matrices.

\item$\ltimes$: MM-STP.

\item $\ttimes$: DK-STP.

\item $I_{m\times n}$: virtual identity in $\ttimes$.

\item $\Pi$: P-STP.

\item $I^P_{m\times n}$: virtual identity in $\Pi$.

\item $\F^{\infty}:=\bigcup_{n=1}^{\infty}\F^{n}$, where $\F$ can be $\R$ or $\C$.








\item $\Col(M)$: set of columns of $M$.

\item $\Col_i(M)$: $i$-th column of $M$.

\item $V_r(A)$: row stacking form of matrix $A$.

\item $V_c(A)$: column stacking form of matrix $A$.

\item $\d_n^i$: $\d_n^i=\Col_i(I_n)$.


\end{itemize}

\vskip 2mm

\section{STP and Ring of Non-square Matrices}

This section is partly based on \cite{che19,che19b,che26}.

\subsection{Semi-Tensor Products}

\begin{dfn}\label{d2.1.1} Given $A\in {\cal M}_{m\times n}$, $B\in {\cal M}_{p\times q}$, and denote $t=n\vee p$.
\begin{itemize}
\item[(i)] The matrix-matrix (MM-) STP is defined by
\begin{align}\label{2.1.1}
A\ltimes B:=(A\otimes I_{t/n})(B\otimes I_{t/p})\in {\cal M}_{(mt/n)\times (qt/p)}.
\end{align}
\item[(ii)] The DK-STP is defined by
\begin{align}\label{2.1.2}
A\ttimes B:=(A\otimes \J^{\T}_{t/n})(B\otimes \J_{t/p})\in {\cal M}_{m\times q}.
\end{align}
\item[(iii)] The project-based (P-) STP is defined by
\begin{align}\label{2.1.3}
A\Pi B:=\frac{n}{t}(A\otimes \J^{\T}_{t/n})(B\otimes \J_{t/p})\in {\cal M}_{m\times q}.
\end{align}
\end{itemize}
\end{dfn}

\begin{prp}\label{p2.1.2}
\begin{itemize}
\item[(i)] Consider DK-STP (\ref{2.1.2}). Then
\begin{align}\label{2.1.4}
A\ttimes B=A\Psi_{n\times p}B,
\end{align}
where
$$
\Psi_{n\times p}=(I_n\otimes \J^{\T}_{t/n})(I_p\otimes \J_{t/p})
$$
is called the bridge matrix of DK-STP.

\item[(ii)] Consider P-STP (\ref{2.1.3}).
\begin{align}\label{2.1.5}
A\Pi B=A\Psi^P_{n\times p}B,
\end{align}
where
$$
\Psi^P_{n\times p}=\frac{n}{t}(I_n\otimes \J^{\T}_{t/n})(I_p\otimes \J_{t/p})
$$
is called the bridge matrix of P-STP.
\end{itemize}
\end{prp}

The followings are some fundamental properties of STP.

\begin{prp}\label{p2.1.3} Let $*\in \{\ltimes,\ttimes,\Pi\}$ be an STP.
\begin{itemize}
\item[(i)]
If $n=p$, then
\begin{align}\label{2.1.8}
A*B=AB.
\end{align}
That is, an STP is a generalization of classical matrix product.

\item[(ii)] (Associativity)
\begin{align}\label{2.1.9}
(A*B)*C=A*(B*C).
\end{align}

\item[(iii)] (Distributivity) Assume $A$ and $B$ have the same size. Then
\begin{align}\label{2.1.11}
\begin{array}{l}
(A+B)*C=A*C+B*C;\\
C*(A+B)=C*A+C*B.
\end{array}
\end{align}
\end{itemize}
\end{prp}

Denote
$$
{\cal M}:=\bigcup_{m=1}^{\infty}\bigcup_{n=1}^{\infty} {\cal M}_{m\times n}.
$$
\begin{cor}\label{c2.1.4}
$({\cal M},*)$ are semi-groups.
\end{cor}

\subsection{Ring of Matrices}

To begin with, we consider the set of matrices with fixed dimension. It is easy to verify the following result.

\begin{prp}\label{p3.1.1} Define
$$
R_{m\times n}:=({\cal M}_{m\times n},\ttimes,+),
$$
and
$$
R^P_{m\times n}:=({\cal M}_{m\times n},\Pi,+).
$$
Then $R_{m\times n}$ (or $R^P_{m\times n}$) is a ring.
\end{prp}

For statement ease, in the following we consider  $R^P_{m\times n}$ only. Similar arguments for  $R_{m\times n}$ also hold.

\begin{dfn}\label{d3.2.1} Consider $R^P_{m\times n}:=({\cal M}_{m\times n},\ttimes,+)$.
\begin{itemize}
\item[(i)]
A virtual identity $I^P_{m\times n}$ is defined to satisfy
\begin{align}\label{3.2.1}
			\begin{aligned}
			&	I^P_{m\times n}\Pi I^P_{m\times n}=I^P_{m\times n},\\
			&	I^P_{m\times n}\Pi A=A\Pi I^P_{m\times n}=A,\quad \forall A\in {\cal M}_{m\times n},\\
			\end{aligned}
\end{align}
\item[(ii)]
		$\overline{{\cal M}}_{m\times n}:=\{aI^P_{m\times n}+A_0\;|\; a\in \R,~A_0\in {\cal M}_{m\times n}\},$
with
\begin{align}\label{3.2.2}
\begin{array}{l}
(aI^P_{m\times n}+A_0)+(bI^P_{m\times n}+B_0)\\
~=(a+b)I^P_{m\times n}+(A_0+B_0),\\
(aI^P_{m\times n}+A_0)\Pi (bI^P_{m\times n}+B_0)\\
~=(ab)I^P_{m\times n}+aB_0+bA_0+A_0\Pi B_0.\\
\end{array}
\end{align}
\end{itemize}
\end{dfn}

Then it is easy to verify the following.

\begin{prp}\label{p3.2.2}
$$
\bar{R}^P_{m\times n}:=\{\overline{{\cal M}}_{m\times n},\Pi,+\}
$$
is a unitary ring, where $\Pi$ and $+$ are defined by (\ref{3.2.2}).
\end{prp}

\subsection{Cross-Dimensional Euclidian Space}
$$
\R^{\infty}:=\bigcup_{n=1}^{\infty}\R^n.
$$

To make it into a topological vector space, we must endow it with (a) a vector space structure and (b) a topological structure on it~\cite{kel63}.

\begin{enumerate}
\item[*] Vector Space Structure:
\end{enumerate}

\begin{dfn}\label{d3.3.1} Consider
$\R^{\infty}$.

Assume $x\in \R^m\subset \R^{\infty}$, $y\in \R^n\subset \R^{\infty}$, $t=\lcm(m,n)$. Define
\begin{itemize}
\item[(i)] (Addition)
\begin{align}\label{3.3.1}
x\hat{\pm}y:=(x\otimes \J_{t/m})\pm (y\otimes \J_{t/n})\in \R^t\subset \R^{\infty}.
\end{align}
\item[(ii)] (Scalar Product)
\begin{align}\label{3.3.2}
a \cdot x:=ax\in \R^{\infty},\quad a\in \R.
\end{align}
\end{itemize}
\end{dfn}

Then it is easy to verify that $(\R^{\infty},\hat{+}, \cdot)$ is a
pseudo-vector space, which satisfies all the requirements of vector space except that the zero is not unique~\cite{abr78}.

\begin{enumerate}
\item[*] Topological Space Structure:
\end{enumerate}

\begin{dfn}\label{d3.3.2}  Assume $x\in \R^m\subset \R^{\infty}$,
$y\in \R^n\subset \R^{\infty}$, $t=\lcm(m,n)$.
\begin{itemize}
\item[(i)] Inner Product (of $x$ and $y$):
\begin{align}\label{3.3.3}
\langle x,y\rangle_{{\cal V}}:=
\frac{1}{t}\langle (x\otimes \J_{t/m}),(y\otimes \J_{t/n})\rangle.
\end{align}
\item[(ii)] Norm (of $x$):
\begin{align}\label{3.3.4}
\|x\|_{{\cal V}}:=\sqrt{\left<x,x\right>_{{\cal V}}}.
\end{align}
\item[(iii)] Distance (of $x$ and $y$):
\begin{align}\label{3.3.5}
d_{{\cal V}}(x,y):=\|x\hat{-} y\|_{{\cal V}}.
\end{align}
\end{itemize}
\end{dfn}

The topology deduced by this distance is called the distance topology. This topological space is  denoted by $(\R^{\infty},{\cal T}_d)$.

\begin{dfn}\label{d3.3.3} Given $x,~y\in \R^{\infty}$. $x$ is said to be equivalent to $y$, denoted by $x\lra y$, if
\begin{align}\label{3.3.6}
d_{{\cal V}}(x,y)=0.
\end{align}
The equivalence class of $x$ is denoted by
$$
\bar{x}=\{y\;|\;y\lra x\}.
$$
\end{dfn}

\begin{prp}[\cite{che19b}]\label{p3.3.4} $x\lra y$, if and only if, there exist $\J_{\a}$ and $\J_{\b}$, such that
\begin{align}\label{3.3.7}
x\otimes \J_{\a}=y\otimes \J_{\b}.
\end{align}
\end{prp}

\begin{dfn}\label{d3.3.5}  Let $\xi\in \R^n$.  $x_0$ is the projection of $\xi$ on $\R^m$, denoted by $x_0=\pi^n_m(\xi)$, defined by:
\begin{align}\label{3.3.8}
x_0=\pi^n_m(\xi):=\argmin_{x\in \R^m}\|\xi \hat{-} x\|_{{\cal V}}.
\end{align}
\end{dfn}

\begin{prp}[\cite{che19b}]\label{p3.3.6} Set $t=\lcm(m,n)$. Then
\begin{align}\label{3.3.9}
\pi^n_m(\xi)=\Pi^n_m\xi,
\end{align}
where
\begin{align}\label{3.3.10}
\Pi^n_m=\frac{m}{t}(I_m\otimes \J^T_{t/m}) (I_n\otimes \J_{t/n}).
\end{align}
Moreover,
\begin{align}\label{3.3.11}
\langle \xi \hat{-} x_0, x_0\rangle_{{\cal V}}=0.
\end{align}
And the Pythagoras' Theorem remains available. That is,
\begin{align}\label{3.3.12}
\|\xi\|_{{\cal V}}^2=\|\xi\vec{-} x_0\|_{{\cal V}}^2+ \|x_0\|_{{\cal V}}^2.
\end{align}
\end{prp}

Figure \ref{Fig.3.1} shows this.

\begin{figure}
\centering
\setlength{\unitlength}{1cm}
\begin{picture}(5,4)\thicklines
\put(0,0){\line(1,0){5}}
\put(0,0){\vector(4,3){4}}
\put(4,3){\vector(0,-1){3}}

\put(1.9,1.9){$\xi$}
\put(2.5,0.1){$x_0$}
\put(4.1,1){$\xi-x_0$}

\put(3.7,0.3){\line(1,0){0.3}}
\put(3.7,0.3){\line(0,-1){0.3}}

\end{picture}
\caption{Projection\label{Fig.3.1}}
\end{figure}
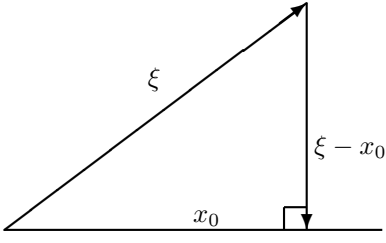

Define the quotient space as
$$
\Omega:=\R^{\infty}/\lra =\{\bar{x}\;|\;x\in \R^{\infty}\}.
$$

\begin{dfn}\label{d3.3.7} Consider $\Omega$. Define
\begin{itemize}
\item[(i)] Inner Product:
\begin{align}\label{3.3.13}
\langle \bar{x},\bar{y}\rangle_{{\cal V}}:=\langle x, y\rangle_{{\cal V}}.
\end{align}
\item[(ii)] Norm:
\begin{align}\label{3.3.14}
\|\bar{x}\|_{{\cal V}}:=\|x\|_{{\cal V}}.
\end{align}
\item[(iii)] Distance:
\begin{align}\label{3.3.15}
d_{{\cal V}}(\bar{x},\bar{y}):=d_{{\cal V}}(x,y).
\end{align}
\end{itemize}
\end{dfn}

\begin{prp}\label{p3.3.8}
\begin{itemize}
\item[(i)] The inner product, norm, and distance defined in Definition \ref{d3.3.7} are properly defined.
\item[(ii)] Under (\ref{3.3.13})--(\ref{3.3.15}), $\Omega$ is a topological vector space, inner product space, and metric space.
\end{itemize}
\end{prp}

 The natural topology on $\R^{\infty}$, denoted by ${\cal T}_n$. All $\R^n$, $n=1,2,\cdots$ are clopen (closed and open) in $\R^{\infty}$.

 \begin{prp}\label{d3.3.9}
Under the  topology ${\cal T}_d$ induced by the distance, $(\R^{\infty}, {\cal T}_d)$
is homeomorphic to $\Omega$. That is
$$
(\R^{\infty},{\cal T}_n)/\lra =(\R^{\infty},{\cal T}_d)\cong \Omega.
$$
\end{prp}

The comparison of the two topological spaces $(\R^{\infty},{\cal T}_n)$ with $(\R^{\infty},{\cal T}_d)$ is depicted by Figure \ref{Fig.3.2}.

\begin{figure}
\centering
\setlength{\unitlength}{5.5mm}
\begin{picture}(9,15)(3.5,-5)
\thicklines
\put(8,3){\line(1,0){7}}
\put(8,3){\line(-1,-1){2}}
\put(13,1){\line(1,1){2}}
\put(13,1){\line(-1,0){7}}
{\color{red}
\put(8,2){\line(1,1){2}}
\put(8,2){\line(1,-1){2}}
\put(12,2){\line(-1,1){2}}
\put(12,2){\line(-1,-1){2}}
}
{\color{blue}
\put(7.8,0){\line(1,1){4}}
}
\put(5.5,0){\line(1,0){9}}
\put(5.5,0){\line(0,1){5}}
\put(5.5,5){\line(1,0){9}}
\put(5.5,5){\line(1,1){2}}
\put(6,4.2){$\Omega^{m\vee n}$}
\put(0,0){\line(1,0){3}}
\put(0,2){\line(1,1){1}}
\put(0,2){\line(1,0){3}}
\put(4,3){\line(-1,-1){1}}
\put(4,3){\line(-1,0){3}}
\put(1,6){\vector(-1,-1){1}}
\put(1,6){\vector(1,0){3}}
\put(1,6){\vector(0,1){1}}
\put(4,3){\line(-1,0){3}}
\put(4.2,0){$\R^1$}
\put(4.2,2.8){$\R^2$}
\put(4.2,6){$\R^s$}
\put(2,3.8){$\vdots$}
\put(13,2.2){$\Omega^n$}
\put(9.3,3.05){$\Omega^m$}
\put(11.5,4.2){$\Omega^{m\wedge n}$}
\put(1,8){$(\R^{\infty}, {\cal T}_n)$}
\put(8,8){$(\R^{\infty}, {\cal T}_d)\cong \Omega$}
\put(5,8){vs}
\end{picture}
\vskip -25mm
\caption{$\R^{\infty}$ ~~vs.~~$\Omega$\label{Fig.3.2}}
\end{figure}
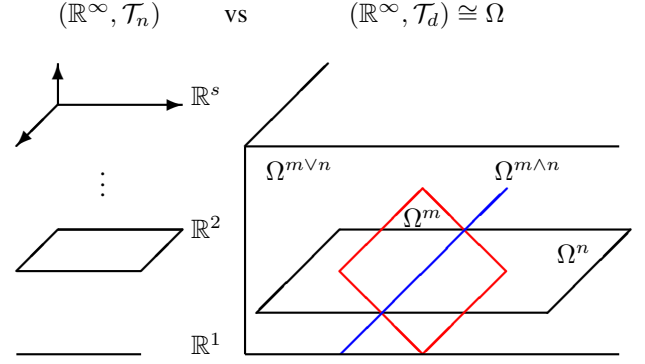

\subsection{Action of $\bar{R}^P_{m\times n}$ on $\R^{\infty}$}

Consider the action of $\bar{R}^P_{m\times n}$ on $\R^{\infty}$.

\begin{dfn}\label{d3.4.1} Let $x\in \R^p$. Define a mapping $\pi: \overline{{\cal M}}_{m\times n}\times \R^{\infty}\ra \R^{\infty}$, denoted by $\hat{\ttimes}$,  as
\begin{align}\label{3.4.1}
A\hat{\ttimes} x:=
\begin{cases}
A\Pi (\Pi^p_m x),& A\in {\cal M}_{m\times n},\\
\Pi^p_m x,       & A=I^P_{m\times n}.
\end{cases}
\end{align}
Then
\begin{align}\label{3.4.2}
(aI^P_{m\times n}+cA_0)\hat{\ttimes} x=a\Pi^p_m x+c A_0\hat{\ttimes} x,\quad x\in \R^p.
\end{align}
\end{dfn}

\begin{prp}\label{p3.4.2} The action defined in Definition \ref{d3.4.1} satisfies the following:
\begin{itemize}
\item[(i)]
\begin{align}\label{3.4.3}
A\hat{\ttimes} (B \hat{\ttimes} x)=(A\Pi B)\hat{\ttimes} x,\quad A,B\in \overline{\cal M}_{m\times n},
\end{align}
\item[(ii)] If $x\lra y$, then
 \begin{align}\label{3.4.4}
A\hat{\ttimes} x=A\hat{\ttimes} y.
\end{align}
\end{itemize}
\end{prp}

\begin{rem}\label{r3.4.3}
\begin{itemize}
\item[(i)] (\ref{3.4.3}) is called the semi-group property, which is fundamental for constructing a dynamic linear system.
\item[(ii)]  (\ref{3.4.4}) makes the mapping consist with the topology ${\cal T}_d$. That is, the operator $\hat{\ttimes}$ is also properly defined on $\overline{{\cal M}}_{m\times n}\times \Omega$.
\end{itemize}
\end{rem}

\section{Hypermatrices}

\subsection{Matricizing and Vectorizing  of Hypermatrices}

The definition of hypermatrix in \cite{lim13} is a mapping from multiple indices to a field as: $\pi: [1,n_1]\times \cdots\times [1,n_d]\ra \F$. It can equivalently be stated in the set sense as follows:

\begin{dfn}\label{d4.1.1}

\begin{itemize}
\item[(i)]
\begin{align}\label{4.1.1}
\begin{array}{r}
{\cal A}=\left\{a_{i_1,\cdots,i_d}\in \F\;|\; i_s\in[1,n_s];~s\in [1,d]\right\},\\
~~\F =\R~(\mbox{or}~ \C),
\end{array}
\end{align}
is a hypermatrix of order $d$ and dimension ${\bf n}=n_1\times\cdots\times n_d$.

The set of hypermatrices of  dimension ${\bf n}:=n_1\times\cdots  \times n_d$ is denoted by
$$
\F^{n_1\times \cdots\times n_d}.
$$

\item[(ii)] A hypermatrix is called a tensor if $n_1=\cdots=n_d$.
\end{itemize}
\end{dfn}

Note that in this paper a tensor is a special hypermatrix. It is also called a cubical hypermatrix~\cite{lim13}.

\begin{dfn}\label{d4.1.2}
\begin{itemize}
\item[(i)] Let ${\cal A}=(a_{i_1,\cdots,i_d})\in \F^{n_1\times \cdots\times n_d}$. The vectorizing of ${\cal A}$ is defined by
\begin{align}\label{4.1.2}
V_{\cal A}=(a_{1,\cdots,1},a_{1,\cdots,2},\cdots, a_{n_1,\cdots,n_d})^{\T}.
\end{align}
\item[(ii)] Let $\vec{i}=(i_1,\cdots,i_d)$, $\vec{j}=(j_1,\cdots,j_s)$, and $\vec{k}=(k_1,\cdots,k_r)$ be three sets of indices, and
$$
\vec{i}=\vec{j}\bigcup \vec{k}
$$
be a partition.
Denote the matricizing of ${\cal A}$ as
$$
A=M^{\vec{j}\times \vec{k}}({\cal A}),
$$
which consists of the entries of ${\cal A}$ arranged in such a way that its rows are labeled by $\vec{j}$ and its columns are labeled by $\vec{k}$.
\end{itemize}
\end{dfn}

\begin{exa}\label{e4.1.3} Assume ${\cal A}=(a_{i,j,k})\in \R^{2\times 3\times 2}$. Then
\begin{itemize}
\item[(i)]
$$
\begin{aligned}
V_{{\cal A}}=(&a_{1,1,1},a_{1,1,2},a_{1,2,1},a_{1,2,2},a_{1,3,1},a_{1,3,2},\\
              &a_{2,1,1},a_{2,1,2},a_{2,2,1},a_{2,2,2},a_{2,3,1},a_{2,3,2})^{\T}.
\end{aligned}
$$
\item[(ii)] Let $\vec{s}=(i,k)$, $\vec{r}=(j)$. Then
$$
M^{\vec{s}\times \vec{r}}({\cal A})=
\begin{bmatrix}
a_{1,1,1}&a_{1,2,1}&a_{1,3,1}\\
a_{1,1,2}&a_{1,2,2}&a_{1,3,2}\\
a_{2,1,1}&a_{2,2,1}&a_{2,3,1}\\
a_{2,1,2}&a_{2,2,2}&a_{2,3,2}\\
\end{bmatrix}.
$$
\end{itemize}
\end{exa}

\subsection{Various EEs}

Theoretically, a tensor of order $d$ has
$$
N=\dsum_{k=0}^d \binom{d}{k}k!(d-k)!
$$
different matricizings. If both $\vec{j}$ and $\vec{k}$ are order preserving, there are
$$
N_0=2^d
$$
order preserving matricizings. $\vec{j}$ is order preserving if $j_k=i_{u_k}$, $k\in [1,s]$, where $s=|\vec{j}|$, then $u_1<u_2<\cdots<u_s$.

For statement ease, hereafter only order preserving partitions are considered. In this case, as long as the partition is determined, the corresponding matricizing is also unique.

Now one sees easily that both the set of data and a matricizing of a hypermatrix can be considered as equivalent expressions of a hypermatrix. But the problem is: when the EEs of a hypermatrix is considered, what is to be concerned? the set of data or its each special matricizings? The answer is obvious. Even in matrix case ($d=2$), the EEs still depend on a particular matrix form. Hence, the EEs of a hypermatrix should be   considered under each special matricizing. This perspective leads to the following definition.

\begin{dfn}\label{d4.2.1} Let ${\cal A}=(a_{i_1,\cdots,i_d})\in \F^{\overbrace{n\times \cdots\times n}^d}$,
$\vec{i}=\vec{j}\cup \vec{k}$ be an (order keeping) partition, and
\begin{align}\label{4.2.1}
A=M^{\vec{j}\times \vec{k}}({\cal A}).
\end{align}
Denote $|\vec{j}|=s$, $|\vec{k}|=r$, $s+r=d$. $u=n^s$ and $v=n^r$. Then $A\in {\cal M}_{u\times v}$.

\begin{itemize}
\item[(i)] $0\neq x\in \C^v$ is an ordinary eigenvector of ${\cal A}$ with respect to partition $\vec{j}\times \vec{k}$, type $B\in {\cal M}_{u\times v}$, and eigenvalue $\l\in \C$, if
\begin{align}\label{4.2.2}
(A-\l B)x=0.
\end{align}
With its corresponding eigenvalue, it becomes an OEE.
\item[(ii)]  $0\neq x\in \C^v$  is a universal eigenvector of ${\cal A}$    with respect to partition $\vec{j}\times \vec{k}$, type
$B=(B_1,\cdots,B_s)$, $B_q\in {\cal M}_{n_{j_q}\times v}$, $q\in [1,s]$, and eigenvalue $\l\in \C$, if
 \begin{align}\label{4.2.3}
 Ax=\l \ltimes_{k=1}^sy_k,
\end{align}
where
$$
\begin{array}{l}
x=\ltimes_{p=1}^rx_p,\quad x_p\in \C^{n_{k_p}},\quad p\in[1,r],\\
y_q=B_qx,\quad q\in [1,s].
\end{array}
$$
With its corresponding eigenvalue, it becomes a UEE.
\item[(iii)] A  universal eigenvector, as defined in (ii), is a diagonal eigenvector, if $n_1=\cdots=n_d$ and
 \begin{align}\label{4.2.4}
\begin{array}{l}
x_1=x_2=\cdots=x_r,\\
y_1=y_2=\cdots=y_s.\\
\end{array}
\end{align}
With its corresponding eigenvalue, it becomes a DEE.
\item[(iv)] A  diagonal eigenvector, as defined in (iii), is a horizontal diagonal eigenvector, if $s=1$ (i.e., $r=d-1$). And
with its corresponding eigenvalue, it becomes a HDEE.
\end{itemize}
\end{dfn}

\begin{rem}\label{r4.2.2}
\begin{itemize}
\item[(i)] OEE comes from the general EE of matrices directly. Its physical meaning is obvious. In other word, its physical meaning is inherited from matrix case.

\item[(ii)] UEE was proposed by \cite{che25}. To see it is a special case of OEE, we denote
$y=Bx$ and assume both $x$ and $y$ in OEE are separable as
$$
y=\ltimes_{p=1}^sy_p,\quad x=\ltimes_{q=1}^rx_q.
$$
Then we have
$$
\ltimes_{p=1}^sy_p=B\ltimes_{q=1}^rx_q.
$$
Using MDA, we can have\footnote{Refer to the next subsection for MDA and the projection mapping $\Xi^e_{[p;{\bf n}]}$, etc.}
$$
\begin{array}{l}
y_p=c_p\Xi^e_{[p;{\bf n}]} B\ltimes_{q=1}^rx_q\\
~~:=B_p\ltimes_{q=1}^rx_q,\quad p\in[1,s],
\end{array}
$$
where $\prod_{p=1}^sc_p=h(y)$, which is the head value of $y$.

Hence the UEE is a special case of OEE.

\item[(iii)] DEE is obviously a special case of UEE.

\item[(iv)] HDEE is a special case of DEE, which is a non-trivial form of matricizing with maximum column number of columns.
It was claimed in \cite{che25} that to their best knowledge all existing definitions of eigenvalues/eigenvectors of hypermatrices up to \cite{che25} are of HDEE. For example, the definitions in \cite{chi13,din15,kol14,li14,qi05,qi08} are all HDEEs.

\item[(v)] One important result of this paper is to provide a general method to find all OEEs of a hypermatrix.  Since
 \begin{align}\label{4.2.5}
\text{OEE}\supset \text{UEE}\supset \text{DEE} \supset \text{HDEE},
\end{align}
the solutions of UEE, DEE, and HDEE can be picked out from the set of OEEs by verifying them, using MDA. (Refer to the next subsection.)
\end{itemize}
\end{rem}

\subsection{Monic Decomposition Algorithm}

This subsection is partly based on \cite{che25}.

\begin{dfn}\label{d4.3.1} A vector $x\in \F^n$ with $n=\prod_{i=1}^d n_i$ is called a hypervector, if there exist $x_i\in \F^{n_i}$, $i\in [1,d]$ such that
\begin{align}\label{4.3.1}
x=\ltimes_{i=1}^dx_i.
\end{align}
Then, $d$ is called the order of $x$ and ${\bf n}=n_1\times \cdots\times n_d$ is the dimension of $x$. The set of hypervectors of
order $d$ and dimension ${\bf n}=n_1\times \cdots\times n_d$ is denoted by
$$
\F^{n_1\ltimes \cdots\ltimes n_d}.
$$
\end{dfn}

If $x$ can be expressed as in (\ref{4.3.1}), then it is said to be decomposable (with respect to ${\bf n}$).
The following result says that
up to a set of product constant coefficients the decomposition is unique.
\begin{prp}\label{p4.3.2} Assume
$$
x=\ltimes_{i=1}^dx_i=\ltimes_{i=1}^dz_i,
$$
then
$$
x_i=c_iz_i,\quad i\in [1,d],
$$
and $\prod_{i=1}^dc_i=1$.
\end{prp}

\begin{dfn}[\cite{che25}]\label{d4.3.3} Let $0\neq x\in \R^n$. The head index, denoted by $e(x)$, is defined by
$$
e(x)=\min\{k\;|\;x_k\neq 0\}.
$$
The head value, denoted by $h(x)$, is
$$
h(x)=x_{e(x)}.
$$
If $h(x)=1$, $x$ is called a monic vector.
\end{dfn}

\begin{rem}\label{r4.3.4} Assume $x=\ltimes_{i=1}^dx_i\in \R^{n_1\ltimes \cdots\ltimes n_d}$. If $e(x)=e$ and $e_i=e(x_i)$, $i\in [1,d]$, then they satisfy the following conversion relation:

\begin{itemize}
\item[(i)] $(e_1,\cdots,e_d)\ra e$:
\begin{align}\label{4.3.2}
\begin{array}{l}
e=(e_1-1)n_2n_3\cdots n_d+(e_2-1)n_3\cdots n_d+\cdots\\
~~~~~+(e_{d-1}-1)n_d+e_d.
\end{array}
\end{align}

\item[(ii)] $e\ra (e_1,\cdots,e_d)$:
\begin{align}\label{4.3.3}
\begin{array}{l}
r_{d+1}:=e-1;\\
r_s=\left[r_{s+1}/n_s\right],\\
c_s=r_{s+1}-n_sr_s,\\
e_s=c_s+1,\quad s=d,d-1,\cdots,2;\\
c_1=r_2,\\
e_1=c_1+1.
\end{array}
\end{align}
\end{itemize}
\end{rem}

\begin{dfn}\label{d2.3.3} Assume $x\in \R^{n}$, where $n=\prod_{i=1}^dn_i$, and $e(x)=e$. Set ${\bf n}=n_1\times \cdots\times n_d$. Define a set of mappings as
\begin{align}\label{4.3.4}
\Xi^e_{[i;{\bf n}]}:=\ltimes_{t=1}^{i-1}[\d_{n_t}^{e_t}]^{\mathrm{T}}\otimes I_{n_i}
\otimes \ltimes_{t=i+1}^{d}[\d_{n_t}^{e_t}]^{\mathrm{T}},\quad i\in [1,d].
\end{align}
\end{dfn}

\begin{thm}[\cite{che25}]\label{t4.3.5} Assume $0\neq x\in \R^{n}$, where $n=\prod_{i=1}^dn_i$,  $e(x)=e$, and $h(x)=h$. Set
$x_0=x/h$, and
\begin{align}\label{4.3.5}
x_i:=\Xi^e_{[i,{\bf n}]}x_0,\quad i\in [1,d].
\end{align}
Then $x$ is decomposable (with respect to ${\bf n}=n_1\times \cdots\times n_d$), if and only if,
\begin{align}\label{4.3.6}
x=h\ltimes_{i=1}^dx_i.
\end{align}
\end{thm}

The algorithm provided by (\ref{4.3.5})--(\ref{4.3.6}) is called the MDA.

\section{KCF vs EEs of Non-Square Matrix}

\subsection{Kronecker Canonical Form of Matrix Pencil}

\begin{dfn}[\cite{gan59}]\label{d5.1.1}
For $A,~B\in {\cal M}_{m\times n}$, The matrix pencil
is defined by
$$
A-\l B=-(\l B-A)\in {\cal M}_{m\times n}.
$$
\end{dfn}

Note that both $A-\l B$ and $\l B-A$ are used in literature for matrix pencil. We also use both for convenience.

The following result is fundamental for Kronecker canonical form of matrix pencil.

\begin{thm}[\cite{gan59}]\label{t5.1.2} Consider a pencil $\l B-A\in {\cal M}_{m\times n}$. There exist two nonsingular matrices
 $P\in {\cal M}_{m\times m}$ and $Q\in {\cal M}_{n\times n}$ such that
\begin{align}\label{5.1.1}
P(\l B-A)Q~=~K(\l),
\end{align}
where $K(\l)$, called the Kronecker  canonical form, is constructed as follows:
\begin{align}\label{5.1.2}
\begin{array}{l}
P(\l B-A)Q= \diag\left(L_{\e_1}(\l),\cdots, L_{\e_p}(\l),L^{\T}_{\eta_1}(\l),\cdots,\right.\\
\left.~L^{\T}_{\eta_q}(\l),N_{\a_1}(\l),\cdots,N_{\a_r}(\l),J^{\l_1}(\l),\cdots, J^{\l_s}(\l)\right),\\
\end{array}
\end{align}
where
$$
\begin{array}{lll}
L_0=\haa, & L_1(\l)=\begin{bmatrix} \l&-1\end{bmatrix},&L_2(\l)=\begin{bmatrix}\l&-1&0\\0&\l&-1\end{bmatrix},\\
\end{array}
$$
$$
\begin{array}{ll}
L_3(\l)=\begin{bmatrix}\l&-1&0&0\\0&\l&-1&0\\0&0&\l&-1\end{bmatrix},&\cdots\\
\end{array}
$$
$$
\begin{array}{lll}
L^{\T}_0=\vaa, & L^{\T}_1(\l)=\begin{bmatrix} -1\\\l\end{bmatrix},&L^{\T}_2(\l)=\begin{bmatrix}-1&0\\\l&-1\\0&\l\end{bmatrix},\\
\end{array}
$$
$$
\begin{array}{ll}
L^{\T}_3(\l)=\begin{bmatrix}-1&0&0\\\l&-1&0\\0&\l&-1\\0&0&\l\end{bmatrix},&\cdots\\
\end{array}
$$

$$
N_1\!=\![-1],\ N_2(\l)\!=\!\!\begin{bmatrix}\! -1&\!0\\\l&\!-1\!\end{bmatrix}\!, \ N_3(\l)\!=\!\!\begin{bmatrix}\!-1&\!0&\!0\\\l&\!-1&\!0\\0&\!\l &\!-1\!\end{bmatrix}\!,\cdots
$$

$$
J^{\l_i}(\l)=\diag\left(J^{\l_i}_{\b_{i,1}}(\l),\cdots, J^{\l_i}_{\b_{i,t_i}}(\l)\right),~i\in[1,s],
$$
and
$$
\begin{array}{ll}
J^{\l_i}_1(\l)=[\l-\l_i], & J^{\l_i}_2(\l)=\begin{bmatrix} \l-\l_i&0\\-1&\l-\l_i\end{bmatrix},\\
\end{array}
$$
$$
\begin{array}{l}
J^{\l_i}_3(\l)=\begin{bmatrix}\l-\l_i&0&0\\-1&\l-\l_i&0\\0&-1&\l-\l_i\end{bmatrix},\cdots\\
i=1,2,\cdots, s.
\end{array}
$$
\end{thm}

Note that in the above $\haa$ stands for several zero columns, and $\vaa$ stands for several zero rows.
We use $|L_0|$ ($|L^{\T}_0|$) for the number of zero columns (rows).

We refer to \cite{van79,bee88} for the calculation of Kronecker canonical form of a matrix pencil. It is worth to emphasize that the calculation involves only (a) elementary row and column operaions; (b) Jordan form transformation of square matrix. Hence, it is straightforward computable (One can use  MatLab fundamental functions).

\subsection{Solving EEs From KCF}

Consider
\begin{align}\label{5.2.1}
(\l B-A)x=0.
\end{align}
Assume $P(\l B-A)Q:=K(\l)$ is the Kronecker canonical form as in (\ref{5.1.2}). Define
$$
z=Q^{-1}x.
$$
Then (\ref{5.2.1}) is equivalent to
 \begin{align}\label{5.2.2}
K(\l)z=0.
\end{align}

According to the block structure of $K(\l)$, $z$ is split to several sub-vectors as follows:
 \begin{align}\label{5.2.3}
z=\begin{bmatrix}
z_{\epsilon}\\
z_{\eta}\\
z_{\a}\\
z_{\b}
\end{bmatrix},
\end{align}
where
$$
z_{\epsilon}\!=\!\!\!\begin{bmatrix}
 z_{\epsilon_1}\\
\vdots\\
z_{\epsilon_p}
\end{bmatrix}\!\!; \;
z_{\eta}\!=\!\!\!\begin{bmatrix}
 z_{\eta_1}\\
\vdots\\
z_{\eta_q}
\end{bmatrix}\!\!; \;
z_{\a}\!=\!\!\!\begin{bmatrix}
 z_{\a_1}\\
\vdots\\
z_{\a_r}
\end{bmatrix}\!\!;\;
z_{\b}\!=\!\!\!\begin{bmatrix}
 z_{\b_1}\\
\vdots\\
z_{\b_s}
\end{bmatrix}\!\!;\;
z_{\b_i}\!=\!\!\!\begin{bmatrix}
 z_{\b_{i,1}}\\
\vdots\\
z_{\b_{i,t_i}}
\end{bmatrix}.
$$
We look for all nonzero solutions $z$ with respect to each diagonal block of $K(\l)$:

\begin{itemize}
\item[(i)] Corresponding to $L_{\epsilon_i}(\l)$.

Assume $\epsilon_i=0$, that is, we have $L_0$ and assume $|L_0|=\ell>0$, which means there are $\ell$ columns of zero. Then the corresponding $0\neq z_{\epsilon_i}\in \C^{\ell}$ can be arbitrary. We call such eigenvector vector-free eigenvector, which means the eigenvector can be arbitrary nonzero complex vector and the corresponding eigenvalue can be arbitrary complex number.

Assume $\epsilon_i=r>0$. Then the corresponding segment solution is
$$
z_{\epsilon_i}=(1,\lambda,\cdots,\lambda^r)^{\T}.
$$
We call such eigenvector the value-free eigenvector, which means the eigenvalue $\l$ can be arbitrary chosen and the corresponding eigenvector is an $\l$-depending polynomial form.

\item[(ii)] Corresponding to $L^{\T}_{\eta_i}(\l)$.

Let $\eta_i=r$. Then it is easy to verify that the corresponding
$$
z_{\eta_i}=0\in \C^r,
$$
which means there is no eigenvalue $\l$ with respect to non-zero solution $z_{\eta_i}$.

\item[(iii)] Corresponding to $N_{\a_i}(\l) $.

Let $\a_i=r$. Then it is easy to verify that similarly to (ii), the corresponding
$$
z_{\a_i}=0\in \C^r.
$$

\item[(iv)] Corresponding to $J^{\l_i}_r(\l)$.

The eigenvalue is $\l_i$ and the corresponding eigenvector is
$$
z_{r}=(0,0,\cdots,0,1)^{\T}\in \C^r.
$$

They are the regular EEs corresponding to the Jordan blocks.
\end{itemize}

The following lemma is obvious.
\begin{lem}\label{l5.2.1}
$0\neq z\in \C^n$ is called an eigenvector of $K(\l)$ with respect to eigenvalue $\l=\l_0$, if
 \begin{align}\label{5.2.301}
K(\l_0)z=0.
\end{align}
$z$ is an eigenvector of $K(\l)$ with respect to the eigenvalue $\l=\l_0$, if and only if, $x=Qz$ is the eigenvector of $A$ with respect to eigenvalue $\l_0$ and type $B$. That is,
 \begin{align}\label{5.2.302}
(\l_0B-A)x=0.
\end{align}
\end{lem}

Using Lemma \ref{l5.2.1}, all possible nonzero solutions $z$ of (\ref{5.2.301}) can be used to  construct nonzero solution $x=Qz$ for (\ref{5.2.302}). The following result is an immediate consequence.

\begin{thm}\label{t5.2.2} Assume the Kronecker canonical form of pencil $\l B-A$ is obtained as (\ref{5.1.2}). According to the diagonal form of $K(\l)$, decompose $Q$ into
\begin{align}\label{5.2.4}
\begin{aligned}
Q=[&Q_{\epsilon_1},\cdots, Q_{\epsilon_p}, Q_{\eta_1},\cdots, Q_{\eta_p},\\
   &Q_{\a_1},\cdots, Q_{\a_r},Q_{\b_1},\cdots, Q_{\b_s}].
\end{aligned}
\end{align}
Then
\begin{itemize}
\item[(i)] Corresponding to each $\epsilon_i=0$, $Q_{\epsilon_i}$ is a set of vector-free eigenvectors corresponding to eigenvalue $\l=0$;
\item[(ii)] Corresponding to each $\epsilon_i>0$, $Q_{\epsilon_i}$ is a  value-free eigenvector corresponding to arbitrary given eigenvalue $\l\in \C$;
\item[(iii)] Corresponding to each Jordan block $J^{\b_i}$, $Q_{\b_i}$ is a chain, containing an eigenvector and $\b_j-1$ root vectors corresponding to eigenvalue $\l_i$.
\end{itemize}
Moreover, all the eigenvalues/eigenvectors of $A$ with respect to $B$ are of the forms (i)--(iii).
\end{thm}

As aforementioned that the Kronecker canonical form is straightforward computable. According to Theorem \ref{t5.2.2}, the general EEs of a matrix $A$ with respect to type $B$, as well as hypermatrices,  are also straightforward computable.

\section{EEs and Canonical Form of Tensor}

\subsection{EEs of Hypermatrix}

First, we consider the OEEs of hypermatrices. As aforementioned, the OEEs of a hypermatrix is considered as EEs of a special matricizing of the hypermatrix, the calculation of OEEs of hypermatrices becomes that of general EEs of matrices, and the problem has been solved in previous section.

As for the UEE of tensor, we can check the set of OEE to see whether each of them is decomposable or not by using MDA. Then all decomposable OEEs  solve the UEE.  Similarly, DEE can be picked out from UEE, and as a special DEE, HDEE can also be obtained.

We give a numerical example to demonstrate this.

\begin{exa}\label{e6.1.1}
Consider ${\cal A}=(a_{i,j.k}),~{\cal B}=(b_{i,j,k})\in \R^{4\times 4\times 14}$, with
$$
A=M^{k\times\{i,j\}}({\cal A}),\quad B=M^{k\times\{i,j\}}({\cal B})
$$
as
$$
\begin{tiny}
A=
\left[
\begin{array}{cccccccccccccccc}
0&2&0&1&0&0&0&0&0&0&0&0&0&0&0&0\\
0&0&0&0&0&1&0&0&0&0&0&0&0&0&0&0\\
0&0&0&0&0&0&1&0&0&0&0&0&0&0&0&0\\
0&0&0&0&0&0&0&0&0&0&0&0&0&0&0&0\\
0&0&0&0&0&0&0&1&0&0&0&0&0&0&0&0\\
0&0&0&0&0&0&0&0&1&0&0&0&0&0&0&0\\
0&5&0&0&0&0&0&0&0&1&0&0&0&0&0&0\\
0&0&0&0&0&0&0&0&0&0&0&0&0&0&0&0\\
0&0&0&0&0&0&0&0&0&0&1&0&0&0&0&0\\
0&0&0&0&0&0&0&0&0&0&0&1&0&-8&0&0\\
0&0&0&0&0&0&0&0&0&0&0&0&1&0&0&0\\
0&0&0&0&0&0&0&0&0&0&0&0&0&2&0&0\\
0&0&0&0&0&0&0&0&0&0&0&0&0&-24&1&-6\\
0&0&0&0&0&0&0&0&0&0&0&0&0&12&1&3\\
\end{array}
\right]
\end{tiny}
$$
$$
\begin{tiny}
B=
\left[
\begin{array}{cccccccccccccccc}
0&10&1&0&0&0&0&0&0&2&0&0&0&0&0&0\\
0&0&  0&0&1&0&0&0&0&0&0&0&0&0&0&0\\
0&0&  0&0&0&1&0&0&0&0&0&0&0&0&0&0\\
0&0&  0&0&0&0&0&0&0&0&0&0&0&0&0&0\\
0&0&  0&0&0&0&0&0&0&0&0&0&0&0&0&0\\
0&0&  0&0&0&0&0&1&0&0&0&0&0&0&0&0\\
0&0&  0&0&0&0&0&0&1&0&0&0&0&0&0&0\\
0&5&  0&0&0&0&0&0&0&1&0&0&0&0&0&0\\
0&0&  0&0&0&0&0&0&0&0&0&0&0&0&0&0\\
0&0&  0&0&0&0&0&0&0&0&0&0&0&-4&0&0\\
0&0&  0&0&0&0&0&0&0&0&0&1&0&0&0&0\\
0&0&  0&0&0&0&0&0&0&0&0&0&0&1&0&0\\
0&0&  0&0&0&0&0&0&0&0&0&0&0&-8&1&-2\\
0&0&  0&0&0&0&0&0&0&0&0&0&0&4&0&1\\
\end{array}
\right]
\end{tiny}
$$
Using standard procedure, we can find $P$ and $Q$ as follows:
 $$
\begin{tiny}
P=
\left[
\begin{array}{cccccccccccccc}
1&0&0&0&0&0&0&-2&0&0&0&0&0&0\\
0&1&0&0&0&0&0&0&0&0&0&0&0&0\\
0&0&1&0&0&0&0&0&0&0&0&0&0&0\\
0&0&0&1&0&0&0&0&0&0&0&0&0&0\\
0&0&0&0&1&0&0&0&0&0&0&0&0&0\\
0&0&0&0&0&1&0&0&0&0&0&0&0&0\\
0&0&0&0&0&0&1&0&0&0&0&0&0&0\\
0&0&0&0&0&0&0&1&0&0&0&0&0&0\\
0&0&0&0&0&0&0&0&1&0&0&0&0&0\\
0&0&0&0&0&0&0&0&0&1&0&4&0&0\\
0&0&0&0&0&0&0&0&0&0&1&0&0&0\\
0&0&0&0&0&0&0&0&0&0&0&1&0&0\\
0&0&0&0&0&0&0&0&0&0&0&0&1&2\\
0&0&0&0&0&0&0&0&0&0&0&0&0&1\\
\end{array}
\right],
\end{tiny}
$$
and
$$
\begin{tiny}
Q=
\left[
\begin{array}{cccccccccccccccc}
1&0 &0&0&0&0&0&0&0&0&0&0&0&0&0&0\\
0&1 &0&0&0&0&0&0&0&0&0&0&0&0&0&0\\
0&0 &1&0&0&0&0&0&0&0&0&0&0&0&0&0\\
0&-2&0&1&0&0&0&0&0&0&0&0&0&0&0&0\\
0&0 &0&0&1&0&0&0&0&0&0&0&0&0&0&0\\
0&0 &0&0&0&1&0&0&0&0&0&0&0&0&0&0\\
0&0 &0&0&0&0&1&0&0&0&0&0&0&0&0&0\\
0&0 &0&0&0&0&0&1&0&0&0&0&0&0&0&0\\
0&0 &0&0&0&0&0&0&1&0&0&0&0&0&0&0\\
0&-5&0&0&0&0&0&0&0&1&0&0&0&0&0&0\\
0&0 &0&0&0&0&0&0&0&0&1&0&0&0&0&0\\
0&0 &0&0&0&0&0&0&0&0&0&1&0&0&0&0\\
0&0 &0&0&0&0&0&0&0&0&0&0&1&0&0&0\\
0&0 &0&0&0&0&0&0&0&0&0&0&0&1&0&0\\
0&0 &0&0&0&0&0&0&0&0&0&0&0&0&1&0\\
0&0 &0&0&0&0&0&0&0&0&0&0&0&-4&0&1\\
\end{array}
\right].
\end{tiny}
$$
And  $P(\l B-A)Q$ has the Kronecker canonical form  (\ref{6.1.1}). 

\begin{figure*}[tb]
\begin{align}\label{6.1.1}
\begin{tiny}
K(\l)=P(\l B-A)Q
=\left[
\begin{array}{cccccccccccccccc}
0&0&\l&-1&0 & 0&0&0&0&  0&0&0&0&0&0&0\\
0&0& 0& 0&\l&-1&0&0&0&  0&0&0&0&0&0&0\\
0&0& 0& 0&0&\l&-1&0&0&  0&0&0&0&0&0&0\\
0&0& 0&0&0&0&  0&  0&0& 0&0&0&0&0&0&0\\
0&0& 0&0&0&0&  0&-1&0& 0&0&0&0&0&0&0\\
0&0& 0&0&0&0&  0&\l&-1& 0&0&0&0&0&0&0\\
0&0&  0&0&0&0&0& 0& \l&-1&0&0&0&0&0&0\\
0&0&  0&0&0&0&0&0&0&  \l&0&0&0&0&0&0\\
0&0&  0&0&0&0&0&0&0&  0&-1&0&0&0&0&0\\
0&0&  0&0&0&0&0&0&0&  0& 0&-1&0&0&0&0\\
0&0&  0&0&0&0&0&0&0&  0& 0& \l&-1&0&0&0\\
0&0&  0&0&0&0&0&0&0&  0& 0& 0& 0&\l-2&0&0\\
0&0&  0&0&0&0&0&0&0&  0& 0& 0& 0&0&\l-3&0\\
0&0&  0&0&0&0&0&0&0&  0& 0& 0& 0&0&-1&\l-3\\
\end{array}
\right].
\end{tiny}
\end{align}
\end{figure*}

Observing (\ref{6.1.1}), we search all non-zero solutions block by block.
\begin{itemize}
\item[(i)]
$$
L_0=\haa,
$$
with $|L_0|=2$.

Then we have two vector-free eigenvectors with respect to eigenvalue $\lambda$ that can
be an arbitrary complex number, which are
$$
\begin{aligned}
x_1&=\Col_1(Q)\\
&=(1,0,0,0,0,0,0,0,0,0,0,0,0,0,0,0)^{\T},\\
x_2&=\Col_2(Q)\\
&=(0,1,0,-2,0,0,0,0,0,-5,0,0,0,0,0,0)^{\T}.
\end{aligned}
$$
\item[(ii)]
$$
L_1=[\l,-1].
$$
Then we have a value-free eigenvector as
$$
\begin{aligned}
x_3&=[\Col_3(Q),\Col_4(Q)]\begin{bmatrix}1\\ \l\end{bmatrix}\\
~&=(0,0,1,0,0,0,0,0,0,0,0,0,0,0,0,0)^{\T}\\
~&~~~~+\l (0,0,0,1,0,0,0,0,0,0,0,0,0,0,0,0)^{\T}.
\end{aligned}
$$
\item[(iii)]
$$
L_2=\begin{bmatrix}
\l&-1&0\\
0&\l&-1\\
\end{bmatrix}.
$$
Then we have another value-free eigenvector as
$$
\begin{aligned}
x_4&=[\Col_5(Q),\Col_6(Q),\Col_7(Q)]\begin{bmatrix}1\\ \l\\\l^2\end{bmatrix}\\
~&=(0,0,0,0,1,0,0,0,0,0,0,0,0,0,0,0)^{\T}\\
~&~~~~+\l (0,0,0,0,0,1,0,0,0,0,0,0,0,0,0,0)^{\T}\\
~&~~~~+\l^2 (0,0,0,0,0,0,1,0,0,0,0,0,0,0,0,0)^{\T}.
\end{aligned}
$$
\item[(iv)]
$$
L_0^{\T}=\vaa,
$$
with $|L_0^{\T}|=1$.
Correspondingly, there is no solution.
\item[(v)]
$$
L_3^{\T}=\begin{bmatrix}
-1&0&0\\
\l&-1&0\\
0&\l&-1\\
0&0&\l
\end{bmatrix}.
$$
Correspondingly, there is no non-zero solution.
\item[(vi)]
$$
N_1=[-1].
$$
Correspondingly, there is no non-zero solution.
\item[(vii)]
$$
N_2=\begin{bmatrix}
-1&0\\
\l&-1
\end{bmatrix}.
$$
Correspondingly, there is no non-zero solution.
\item[(viii)]
$$
J^2_1=[\l-2].
$$
We have a regular eigenvector
$$
\begin{aligned}
x_5&=\Col_{14}(Q)\\
~&=(0,0,0,0,0,0,0,0,0,0,0,0,0,1,0,-4)^{\T},
\end{aligned}
$$
with respect to eigenvalue $\l_0=2$.
\item[(ix)]
$$
J^3_2=\begin{bmatrix}
\l-3&0\\
-1&\l-3
\end{bmatrix}.
$$
We have  another regular eigenvector
$$
\begin{aligned}
x_6&=\Col_{16}(Q)\\
~&=(0,0,0,0,0,0,0,0,0,0,0,0,0,0,0,1)^{\T},
\end{aligned}
$$
with respect to eigenvalue $\l_0=3$.
\end{itemize}
\end{exa}

\begin{rem}\label{r6.1.2} Recall Example \ref{e6.1.1}.
\begin{itemize}
\item[(i)] There are $6$  sets of EEs described by $x_1,\cdots,x_6$. A straightforward verification shows that they satisfy (\ref{5.2.302}).

\item[(ii)] Corresponding to block $J^3_2$, one sees easily that
$$
x_7=\Col_{15}(Q)
$$
is a ``root vector", satisfying
$$
P(\l_0 B-A)x_7+\Phi x_6=0,
$$
where $\l_0=3$, $\Phi=\diag(0_{12\times 14},I_2)$ is the mapping which draws out the last block of $K(\l)$.

Since in non-square case the chain of eigenvector with root vectors does not form an $A$-invariant sub-space, they are of less interest.

\end{itemize}
\end{rem}

Example \ref{e6.1.1} reveals all the OEEs of the hypermatrix $A$ with respect to type $B$.  The including relation (\ref{4.2.5})   shows that we can find UEE, DEE, and HDEE sequentially. The following example shows this. Note that if block-vice eigenvectors with different blocks are corresponding to the same eigenvalue, their combinations remain to be the eigenvectors of the eigenvalue. Particularly, value-free eigenvectors can be used to combined with any eigenvectors as long as the free eigenvalues can be chosen as the same with any given eigenvector.

\begin{exa}\label{e6.1.3} Recall Example \ref{e6.1.1}.
\begin{itemize}
\item[(i)] Consider $x_1$:
Since
$$
x_1=z_1\ltimes z_1,
$$
with $z_1=(1,0,0,0)^{\T}$, it is UEE and DEE.

Note that since ${\cal A}$ is of order $3$ and $A$, $B$ are of the standard form (i.e., one row index and two column indices), each DEE is HDEE. So for this example we do not need to mention HDEE anymore.

\item[(ii)] Consider $x_2$:
Using MDA, it is easy to verify that $x_2$ is not separable.

\item[(iii)] Consider $x_3$:
Since for any $\l$ we have
$$
x_3=(1,0,0,0)^{\T}\ltimes (0,0,1,\l)^{\T},
$$
it is UEE for arbitrary $\l$. But it is not DEE.

Set $\l=0$, we have
$$
x_3^0=(1,0,0,0)^{\T}\ltimes (0,0,1,0)^{\T},
$$
Combining with $x_1$, one sees easily that
$$
x=ax_1+bx_3^0=(1,0,0,0)^{\T}\ltimes (a,0,b,0)^{\T}
$$
is a general form of UEE, with respect to $\l=0$, where $|a|+|b|>0$, (i.e., at least one of $a$ and $b$ is non-zero).

\item[(iv)] Consider $x_4$:
$$
x_4=(0,1,0,0)^{\T}\ltimes (1,\l,\l^2,0)^{\T},
$$
it is UEE for arbitrary $\l$. But it is not DEE.

Set $\l=0$, we have
$$
x_4^0=(0,1,0,0)^{\T}\ltimes (1,0,0,0)^{\T},
$$
Combining with $x_1$, one sees easily that
$$
x=ax_1+bx_4^0=(a,b,0,0)^{\T}\ltimes (1,0,0,0)^{\T}
$$
is another set of UEE with respect to $\l=0$.

\item[(v)] Consider $x_5$:
We have
$$
x_5=(0,0,0,1)^{\T}\ltimes (0,1,0,-4)^{\T},
$$
which is UEE with respect to $\l=2$.

If we consider the OEE with respect to $\l=2$, then
$$
x_3^2:=x_3|_{\l=2}=(1,0,0,0)^{\T}\ltimes (0,0,1,2)^{\T},
$$
and
$$
x_4^2:=x_4|_{\l=2}=(0,1,0,0)^{\T}\ltimes (1,2,4,0)^{\T},
$$
are also such OEEs. Hence the general OEE with respect to $\l=2$ is
$$
x=ax_5+bx_3^2+cx_4^2,
$$
where $|a|+|b|+|c|>0$.
\item[(vi)] Consider $x_6$:
Since
$$
x_6=(0,0,0,1)^{\T}\ltimes (0,0,0,1)^{\T},
$$
it is DEE  with respect to $\l=3$.

If we consider the OEE with respect to $\l=3$, then
$$
x_3^3:=x_3|_{\l=3}=(1,0,0,0)^{\T}\ltimes (0,0,1,3)^{\T},
$$
and
$$
x_4^3:=x_4|_{\l=3}=(0,1,0,0)^{\T}\ltimes (1,3,9,0)^{\T},
$$
are also such OEEs. Hence the general OEE with respect to $\l=3$ is
$$
x=ax_6+bx_3^3+cx_4^3.
$$
\end{itemize}
\end{exa}

\subsection{Canonical Form of Tensor}

Corresponding to the Jordan canonical form of square matrices, we would like to present a canonical form for tensors. Consider ${\cal A}=(a_{i_1,\cdots,i_d})\in \F^{\overbrace{n\times\cdots\times n}^d}$, where $\vec{i}=\vec{j}\cup \vec{k}$ is a partition with
$|\vec{j}|=s$, $|\vec{k}|=r$ and $s+r=d$. Let $u=n^s$, $v=n^r$, and
\begin{align}\label{6.2.1}
A=M^{\vec{j}\times \vec{k}}({\cal A})\in {\cal M}_{u\times v}.
\end{align}
We look for a canonical form of $A$, called the canonical form of tensor ${\cal A}$ with respect to the partition $\vec{i}=\vec{j}
\cup \vec{k}$.

In matrix case, it is easy to see that the Jordan canonical form can be produced from the Kronecker canonical form
of pencil $K(\l)=P(A-\l B)Q$, where $B=I_n$ and $P=Q^{-1}$  by setting $\l=0$. That is, the Jordan canonical form of a square matrix
$A\in {\cal M}_{n\times n}$ is the Kronecker canonical form $K(0)$.
Motivated by this, we give the following definition.

\begin{dfn}\label{d6.2.1} Let $A$ be a special matricizing of tensor  ${\cal A}=(a_{i_1,\cdots,i_d})\in \F^{\overbrace{n\times\cdots\times n}^d}$ with respect to the partition $\vec{i}=\vec{j}\cup \vec{k}$ and
\begin{align}\label{6.2.2}
K(\l)=P(\l B-A)Q
\end{align}
is the Kronecker canonical form of pencil $\l B-A$ with $B=I^P_{u\times v}$. Then
\begin{align}\label{6.2.3}
-K(0)\sim A
\end{align}
is called the Kronecker canonical form of ${\cal A}$ (with respect to  the partition $\vec{i}=\vec{j}\cup \vec{k}$ ).
\end{dfn}

\begin{rem}\label{r6.2.2}
\begin{itemize}
\item[(i)] Taking (\ref{3.4.1}) into consideration, we can simply replace $I^P_{u\times v}$ by $\Psi^P_{u\times v}=\Pi^v_u$, because
$$
(\l I^P_{u\times v}-A)x=0,\quad x\in \F^n
$$
is equivalent to
$$
(\l \Pi^v_{u}-A)x=0,\quad x\in \F^n.
$$
\item[(ii)] Since $I^P_{u\times v}$ is a generalization of $I_u$ (as $u=v$, $I^P_{u\times v}=I_u$), Definition \ref{d6.2.1} is a generalization of the matrix case.
\end{itemize}
\end{rem}

\begin{prp}\label{p6.2.3} Consider the Kronecker canonical form of $\l B-A$.
\begin{itemize}
\item[(i)] If $B$ is of full row rank, there are no $L^{\T}_{\eta}(\l)$ and $N_{\a}(\l)$.

\item[(ii)] If $B$ is of full column rank, there are no $L_{\epsilon}(\l)$ and $N_{\a}(\l)$.

\end{itemize}
\end{prp}

\noindent{\it Proof.} Assume there exists a $L^{\T}_{\eta}(\l)$ (or a $N_{\a}(\l)$), then it is easy to see that
$B$ has at least one zero row. Hence it can not be of full row rank. Similarly for (ii).
\hfill $\Box$

Consider a tensor ${\cal A}=(a_{i_1,\cdots,i_d})\in \R^{\overbrace{n\times \cdots\times n}^d}$, $\vec{i}=\vec{j}\cup \vec{k}$,
$|\vec{j}|=s$, $|\vec{k}|=r$,  $u=n^s$, and $v=n^r$.
$$
A=M^{\vec{j}\times \vec{k}}({\cal A}).
$$
Then
\begin{align}\label{6.2.4}
B=\begin{cases}
\frac{1}{n^{r-s}}(I_{n^s}\otimes \J^{\T}_{n^{r-s}}),& r\geq s,\\
I_r\otimes \J_{n^{s-r}},& r<s.
\end{cases}
\end{align}

\begin{rem}\label{r6.2.4} Observing (\ref{6.2.4}), the following are obvious.
\begin{itemize}
\item[(i)] If $r=s$, $B=I_{n^s}$. This fact shows that when $r=s$ the Kronecker canonical form degenerates to Jordan canonical form.
\item[(ii)] When $r>s$, $B$ is of full row rank; and when $r<s$,  $B$ is of full column rank.
 \end{itemize}
 \end{rem}

\begin{exa}\label{e6.2.5}

\begin{enumerate}
\item[(1)]

We first do some preparation to simplify the compression.

\begin{itemize}
\item[(i)] Consider a tensor  ${\cal A}=(a_{i_1,i_2,i_3})\in \R^{n\times n\times n}$ and let  $\vec{j}=\{i_1\}$ and $\vec{k}=\{i_2,i_3\}$.
To simplify the calculation, we set
$$
B=I_n\otimes \J^{\T}_n=\Psi_{n\times n^2}.
$$
Then the corresponding eigenvalue
$$
\tilde{\l}:=\frac{1}{n}\l.
$$
\item[(ii)]
Define
$$
\begin{array}{l}
Q_0:=\d_{n^2}[1,n+1,\cdots,(n-1)n+1;\\
~\left<2\right>-\left<1\right>,\cdots,\left<n\right>-\left<1\right>; \\
~\left<n+2\right>-\left<n+1\right>,\cdots,\left<2n\right>-\left<n+1\right>;\\
~\cdots;\left<(n-1)n+2\right>-\left<(n-1)n+1\right>,\\
~\cdots,\left<n^2\right>-\left<(n-1)n+1\right>],
\end{array}
$$
where
$\d_t(\left<p\right>-\left<q\right>):=\d_t^p-\d_t^q$.
Then
$$
B*Q_0=[I_n,0_{n\times (n-1)n}].
$$
This elementary column operation can save  many times of column compressions of $B$.
\end{itemize}

\item[(2)]

Consider $A_0=M^{i_1\times i_2i_3}({\cal A})$ as
$$
A_0=\begin{bmatrix}
2&3&2&0&0&0&0&0&0\\
0&-1&0&0&0&0&2&2&2\\
0&1&0&1&1&1&-1&-1&-1\\
\end{bmatrix}.
$$
We look for its KCF.

Using $Q_0$, we have
$$
A=A_0Q_0=\begin{bmatrix}
2&0&0&1&0&0&0&0&0\\
0&0&2&-1&0&0&0&0&0\\
0&1&-1&1&0&0&0&0&0\\
\end{bmatrix}.
$$
$$
B_0=I_3\otimes \J^{\T}_3.
$$
$$
B=B_0Q_0=[I_3,0_{3\times 6}].
$$

Then using the following elementary transformations
$$
\begin{array}{l}
c_1\ra c_1+0.25c_3,\\
r_3\ra r_3-0.25 r_1,\\
c_1\ra c_1+0.25 c_4,\\
c_2\ra c_2-2c_3,\\
r_3\ra r_3+2r_2,\\
c_3\ra c_3-2c_4,
\end{array}
$$
and summarizing them, we have
$$
P=\begin{bmatrix}0&1&1\\
0.25&0&-0.25\\
0.5&-0.25&0.25
\end{bmatrix};
$$
$$
Q=\begin{bmatrix}
0.5&6&1&-1&-1&0&0&0&0\\
-0.25&-4&0&1&0&0&0&0&0\\
0&0&0&0&1&0&0&0&0\\
0.75&2&-1&0&0&-1&-1&0&0\\
0&0&0&0&0&1&0&0&0\\
0&0&0&0&0&0&1&0&0\\
0.25&-2&1&0&0&0&0&-1&-1\\
0&0&0&0&0&0&0&1&0\\
0&0&0&0&0&0&0&0&1\\
\end{bmatrix}.
$$
Then
\begin{align}\label{4.101}
\begin{array}{l}
K(\tilde{\l}):=P(A-\tilde{\l} B)Q=\\
\begin{bmatrix}
1-\tilde{\l} &0&0&0&0&0&0&0&0\\
0 &-\tilde{\l}&1&0&0&0&0&0&0\\
0&0&-\tilde{\l}&1&0&0&0&0&0\\
\end{bmatrix}.
\end{array}
\end{align}
Hence the KCF of $A$ is
$$
A\sim K(0)=
\begin{bmatrix}
1&0&0&0&0&0&0&0&0\\
0 &0&1&0&0&0&0&0&0\\
0&0&0&1&0&0&0&0&0\\
\end{bmatrix}.
$$
Solving
$$
K(\tilde{\l})y=0,
$$
and letting
$$
y=\begin{bmatrix}y^1\\y^2\\y^3\end{bmatrix},~y^1\in \R, ~y^2\in \R^3,~y^3\in \R^5.
$$

$$
y^1=1,~y^2=\begin{bmatrix}1\\\tilde{\l}\\{\tilde{\l}}^2\end{bmatrix},~y^3\in \R^5,
$$
 we have
\begin{itemize}
\item[(1)] Regular Eigenvector:
$$
y=\begin{bmatrix}1\\0_3\\0_5\end{bmatrix},
$$
corresponding to eigenvalue $\tilde{\l}=1$.

\item[(2)] Free $\l$ Eigenvector:
$$
y=\begin{bmatrix}0\\y^2\\0_5\end{bmatrix}, \mbox{where}~ y^2=\begin{bmatrix}1\\\tilde{\l}\\{\tilde{\l}}^2\end{bmatrix}.
$$
\end{itemize}

Then we can choose three different $\tilde{\l}$'s to get three corresponding eigenvectors.

\begin{itemize}
\item[(3)] Free vector Eigenvector:
$$
y=\begin{bmatrix}0\\0_3\\y^3\end{bmatrix}, \mbox{where}~ y^3\in \R^5 ~\mbox{is arbitrary}.
$$

And find linearly independent $y^3$'s, which correspond to $\tilde{\l}=0$.
\end{itemize}
Note that in fact we have
$$
(A-\tilde{\l} B)Qy=0.
$$
And what we do need is to find $\l$ and $x$, satisfying
$$
(A-{\l} B)x=0.
$$
Hence we have $x=Qy$ is the eigenvector with respect to $\l=3\tilde{\l}$.
$$
y^1=1,~y^2=\begin{bmatrix}1\\\tilde{\l}\\{\tilde{\l}}^2\end{bmatrix},~y^3\in \R^5.
$$
Hence, we have
\begin{itemize}
\item[(i)] Regular Eigenvector:
$$
x=(0.5,-0.25,0,0.75,0,0,0.25,0,0)^{\T}
$$
with respect to eigenvalue $\l=3$.

\end{itemize}

\begin{itemize}
\item[(2)] Value-free  Eigenvector:
$$
\begin{aligned}
x=~&(6,-4,0,2,0,0,-2,0,0)^{\T}\\
~&+\frac{\l}{3}(1,0,0,-1,0,0,1,0,0)^{\T}\\
~&+\frac{{\l}^2}{9}(-1,1,0,0,0,0,0,0,0)^{\T}.
\end{aligned}
$$
Choosing $\l_1\neq \l_2\neq \l_3$ arbitrary, the corresponding $x(\l)$ are three eigenvectors with respect to $\l_1$, $\l_2$ and $\l_3$ respectively.

\item[(3)] Vector-free  Eigenvector:
One may choose
$$
x_i:=\Col_i(Q), \quad i=5,6,7,8,9,
$$
as $5$ free vector eigenvectors with respect to $\l=0$.

\end{itemize}

Note that $\Col(Q)$ consists of (ordinary) eigenvectors of $A$!
\end{enumerate}
\end{exa}

\begin{exa}\label{e6.2.6}

Let ${\cal A}$ be the same as in Example \ref{e6.2.5}.

Consider $A_0=M^{i_1i_2\times i_3}({\cal A})$ as
$$
A_0=\begin{bmatrix}
2&3&2\\
0&0&0\\
0&0&0\\
0&-1&0\\
0&0&0\\
2&2&2\\
0&1&0\\
1&1&1\\
-1&-1&-1\\
\end{bmatrix},
$$
$$
B_0=\Pi^3_9=I_3\otimes \J_3.
$$
We look for its KCF.

Skipping the standard computation, we have
$$
P=\begin{bmatrix}
1&0&0&3&-3&0&2&-2&0\\
0&0&0&0&1&0&0&0&0\\
0&0&0&-1&1&0&0&0&0\\
1&0&0&4&-4&0&3&-2&0\\
0&0&0&0&0&0&-1&1&0\\
-1&1&0&-3&3&0&-2&2&0\\
0&0&0&2&-3&1&2&-2&0\\
-1&0&1&-3&3&0&-2&2&0\\
0&0&0&-2&2&0&-2&1&1\\
\end{bmatrix},
$$
$$
Q=\begin{bmatrix}
1&0&0\\
0&1&0\\
-1&0&1\\
\end{bmatrix},
$$
$$
K(\l)=P(A_0-\l B_0)Q=\begin{bmatrix}
-\l&0&0\\
0&-\l&0\\
0&1&0\\
0&0&-\l\\
0&0&1\\
0&0&0\\
0&0&0\\
0&0&0\\
0&0&0\\
\end{bmatrix}.
$$
Then the KCF of $A_0$ is
$$
A_0\sim K(0)=\begin{bmatrix}
0&0&0\\
0&0&0\\
0&1&0\\
0&0&0\\
0&0&1\\
0&0&0\\
0&0&0\\
0&0&0\\
0&0&0\\
\end{bmatrix}.
$$

The only eigenvalue of $A_0$ is $\l=0$, its corresponding eigenvector is
$$
\Col_1(Q)=(1,0,-1)^{\T}.
$$
\end{exa}

\section{Equivalence vs Eigenvalue/Eigenvector}

\subsection{Eigenvalue/Eigenvector of Equivalent Hypermatrix}

\begin{dfn}\label{d7.1.1}
\begin{itemize}
\item[(i)] \cite{gan59} Pencils $A+\l B\in {\cal M}_{m\times n}$ and $A_1+\l B_1\in {\cal M}_{m\times n}$ are strictly equivalent, if there exist
nonsingular $P\in {\cal M}_{m\times m}$ and $Q\in {\cal M}_{n\times n}$, such that
\begin{align}\label{7.1.1}
P(A+\l B)Q=A_1+\l B_1.
\end{align}

\item[(ii)]
Let ${\cal A}=(a_{i_1,\cdots,i_d})\in \R^{\overbrace{n\times \cdots \times n}^d}$, ${\cal C}=(c_{i_1,\cdots,i_d})\in \R^{\overbrace{n\times\cdots\times n}^d}$, and $\vec{i}=\vec{j}\cup \vec{k}$ be a partition of index set with $|\vec{j}|=s$, $|\vec{k}|=r$ and $s+r=d$.
${\cal A}$ and ${\cal B}$ are said to be equivalent with respect to the partition, denoted by
$$
{\cal A}\sim_{\vec{j}\times \vec{k}}{\cal C},
$$
if
$$
A:=M^{\vec{j}\times \vec{k}}({\cal A})~\mbox{and}~C:=M^{\vec{j}\times \vec{k}}({\cal C})
$$
are equivalent. That is, there exist non-singular $P\in {\cal M}_{u\times u}$ and $Q\in {\cal M}_{v\times v}$, where $u=n^s$ and $v=n^r$, such that
\begin{align}\label{7.1.2}
PAQ=C.
\end{align}
\end{itemize}
\end{dfn}

The following result is straightforward verifiable.

\begin{prp}\label{p7.1.2} Assume $A$ and $C$ are the matricizings of ${\cal A}, {\cal C} \in \R^{\overbrace{n\times \cdots\times n}^d}$ respectively. Moreover, $A$ and $C$ are equivalent, that is, (\ref{7.1.2}) holds.
Then for any $B\in {\cal M}_{u\times v}$, the pencil $A+\l B$ is strictly equivalent to $C+\l B'$, where $B'=PBQ$.
\end{prp}

As an immediate consequence, we have the following result.

\begin{cor}\label{c7.2.3} Assume $A$ and $C$ are equivalent such that (\ref{7.1.2}) holds. Assume $x$ is an ordinary eigenvector of $C$ with respect to eigenvalue $\l$ and type $B$, then $z=Q^{-1}x$ is an ordinary eigenvector of $A$ with respect to eigenvalue $\l$ and type $B'=PBQ$.
\end{cor}

\subsection{Equivalence Propagation}

Let ${\cal A}=(a_{i_1,\cdots,i_d})\in \R^{\overbrace{n\times\cdots\times n}^d}$, ${\cal C}=(c_{i_1,\cdots,i_d})\in \R^{\overbrace{n\times\cdots\times n}^d}$, and
\begin{align}\label{7.2.1}
{\cal A}\sim_{\vec{j}\times \vec{k}}{\cal C}.
\end{align}
Assume $\vec{i}=\vec{u}\cup \vec{v}$ is another partition. We ask when
\begin{align}\label{7.2.2}
{\cal A}\sim_{\vec{u}\times \vec{v}}{\cal C}?
\end{align}

The following result is obvious.

\begin{prp}\label{p7.2.1} If (\ref{7.2.1}) holds, then
\begin{align}\label{7.2.3}
{\cal A}\sim_{\vec{k}\times \vec{j}}{\cal C}.
\end{align}
\end{prp}

\noindent{\it Proof.}
Suppose
$$
A:=M^{\vec{j}\times \vec{k}}({\cal A})~\mbox{and}~C:=M^{\vec{j}\times \vec{k}}({\cal C}),
$$
and there exist nonsingular $P$ and $Q$ such that (\ref{7.1.2}) holds.
Since
$$
A^{\T}=M^{\vec{k}\times \vec{j}}({\cal A})~\mbox{and}~C^{\T}=M^{\vec{k}\times \vec{j}}({\cal C}),
$$
taking transpose on both sides of (\ref{7.1.2}) demonstrates the claim.
\hfill $\Box$

Consider the general case. We need a lemma first.

\begin{lem}[\cite{che20}]\label{l7.2.2} Let $A\in {\cal M}_{m\times n}$,  $X\in {\cal M}_{n\times p}$, and $B\in {\cal M}_{p\times q}$. Then
 \begin{align}\label{7.2.4}
\begin{array}{l}
V_c(AXB)=(B^{\T}\otimes A)V_c(X);\\
V_r(AXB)=(A\otimes B^{\T})V_r(X).\\
\end{array}
\end{align}
\end{lem}

Denote
$$
\begin{array}{ll}
A:=M^{\vec{j}\times \vec{k}}({\cal A})~\mbox{and},&C:=M^{\vec{j}\times \vec{k}}({\cal C}),\\
\tilde{A}:=M^{\vec{u}\times \vec{v}}({\cal A})~\mbox{and},&\tilde{C}:=M^{\vec{u}\times \vec{v}}({\cal C}).\\
\end{array}
$$
We assume there exist nonsingular $P\in {\cal M}_{n^s\times n^s}$ and $Q\in {\cal M}_{n^r\times n^r}$, such that
 \begin{align}\label{7.2.5}
PAQ=C.
\end{align}
Using Lemma \ref{l7.2.2}, (\ref{7.2.5}) is equivalent to
 \begin{align}\label{7.2.6}
(P\otimes Q^{\T})V_r(A)=V_r(C).
\end{align}

Assume $|\vec{u}|=\a$, $|\vec{v}|=\b$, then $\a+\b=d$. We ask when
there exist nonsingular $R\in {\cal M}_{n^{\a}\times n^{\a}}$ and $S\in {\cal M}_{n^{\b}\times n^{\b}}$, such that
 \begin{align}\label{7.2.7}
(R\otimes S^{\T})V_r(\tilde{A})=V_r(\tilde{C})?
\end{align}
Note that in $\V_r(A)$ the elements of ${\cal A}$ are arranged by the set of indices $j_1\cdots,j_s,k_1,\cdots,k_r$, and
in $\V_r(\tilde{A})$ the elements of ${\cal A}$ are arranged by the set of indices $u_1\cdots,u_{\a},v_1,\cdots,v_{\b}$. Both
$j_1\cdots,j_s,k_1,\cdots,k_r$ and $u_1\cdots,u_{\a},v_1,\cdots,v_{\b}$ are (two) permutations of $i_1,\cdots,i_d$. Hence, we have
$$
\begin{array}{l}
(j_1\cdots,j_s,k_1,\cdots,k_r)=(i_{p_1},\cdots,i_{p_d}),\\
(u_1\cdots,u_{\a},v_1,\cdots,v_{\b})=(i_{q_1},\cdots,i_{q_d}).
\end{array}
$$
Let $\sigma \in {\bf S}_{n^d}$ be a permutation, defined by
$$
\sigma: (p_1,\cdots,p_d)\mapsto (q_1,\cdots,q_d),
$$
which can be re-arranged as
$$
\sigma: (1,2,\cdots,d)\mapsto (\sigma(1),\cdots,\sigma(d)).
$$
Then \cite{che26} we can construct a permutation matrix $W:=W^{\sigma}_{[\underbrace{n\cdots,n}_d]}$ such that\footnote{The {\bf STP Toolbox} can be found at \url{https://lsc.amss.ac.cn/~hsqi/soft/STP.zip}. Then a function can be found that
$$
	W^{\sigma}_{[\underbrace{n,\cdots,n}_d]}=\mathrm{npermat}([\overbrace{n,\cdots,n}^d],[\sigma(1),\cdots,\sigma(d)]).
$$
}
 \begin{align}\label{7.2.8}
WV_r(A)=V_r(\tilde{A}),\quad WV_r(C)=V_r(\tilde{C}).
\end{align}

Using (\ref{7.2.8}), (\ref{7.2.7}) is equivalent to
 \begin{align}\label{7.2.9}
W^{-1}(R\otimes S^{\T})WV_r(A)=V_r(C).
\end{align}
Comparing with (\ref{7.2.6}), one sees easily that (\ref{7.2.7}) holds, if and only if, there exist nonsingular $R$ and $S$, such that
\begin{align}\label{7.2.10}
W^{-1}\left[ (R\otimes S^{\T})-W(P\times Q^{\T})W^{-1}\right]WV_r(A)=0.
\end{align}

Summarizing the above argument, we have the following result.

\begin{thm}\label{t7.2.3} Using the above notations, ${\cal A}\sim_{\vec{j}\times \vec{k}}$ can propagate to
${\cal A}\sim_{\vec{u}\times \vec{v}}$, if and only if,
$W(P\times Q^{\T})W^{-1}$ has an approximate KPD with respect to $n^{2\a}\times n^{2\b}$ as
\begin{align}\label{7.2.11}
W(P\times Q^{\T})W^{-1}-R\times S^{-1}=E,
\end{align}
where $R\in {\cal M}_{n^{\a}\times n^{\a}}$ and $S\in {\cal M}_{n^{\b}\times n^{\b}}$ are non-singular.
Moreover,
\begin{align}\label{7.2.12}
W^{-1}EW\in \ker(V_r(A)).
\end{align}
\end{thm}

To use Theorem \ref{t7.2.3}, we propose the following algorithm, which is sufficient for the equivalence propagation.

\begin{alg}\label{a7.2.4}
\begin{itemize}
\item[] Step 1:

Calculate
$$
X:=V_r\left[W(P\otimes Q^{\T})W^{-1}\right].
$$
\item[] Step 2:
Using MDA to calculate
$$
X_1:=\Xi^e_{[1;n^{2\a}\times n^{2\b}]}X;\quad X_2:=\Xi^e_{[2;n^{2\a}\times n^{2\b}]}X.
$$
\item[] Step 3:
$$
R=V_r^{-1}(X_1),\quad S^{-1}=V_r^{-1}(X_2).
$$
(Here we assume both $V_r^{-1}(X_1)$ and $V_r^{-1}(X_2)$ are non-singular. Otherwise, the algorithm terminates.)
\item[] Step 4:
Calculate
$$
E=W(P\otimes Q^{\T})W^{-1}-R\otimes S^{-1}.
$$
\item[] Step 5:
Verify if
$$
W^{-1}EW\in \ker(V_r(A)).
$$
If ``Yes" the propagation is successful. Otherwise, the algorithm false.
\end{itemize}
\end{alg}

\section{Conclusion}

Four kinds of eigenvalues/eigenvectors of hypermatrices, namely, OEE, UEE, DEE, and HDEE, are proposed via their matricizings. Then the KCF of matrix pencils is used to derive the KCF of hypermatrices. The OEE of hypermatrices can be obtained from their KCF, and KCF is straghtforward combutable by (a) elementary row/column operations; and (b) Jordan canonical form calculation of square matrices. Then UEE, DEE, and HDEE, as subsets of OEE, can be picked out from OEE using MDA. Hence, this paper provides a comprehensive algorithm for calculating all the eigenvalues/eigenvectors of hypermatrices. In addition, using virtual identity a canonical form of tensors. called the KCF of tensors, is proposed. It is also straightforward computable. The KCF, containing Jordan canonical form as its special case, reveals the relationship between eigenvalues and eigenvectors of a tensor.

\bibliographystyle{IEEEtran}
\bibliography{my}

\end{document}